\documentclass[11pt]{article}

\usepackage[T1]{fontenc}
\usepackage[utf8]{inputenc}
\usepackage[a4paper,margin=1in]{geometry}
\usepackage{amsmath,amssymb,amsthm}
\usepackage{aliascnt}
\usepackage{microtype}
\usepackage[hidelinks]{hyperref}
\usepackage[nameinlink,noabbrev]{cleveref}

\allowdisplaybreaks

\newtheorem{theorem}{Theorem}[section]

\newaliascnt{lemma}{theorem}
\newtheorem{lemma}[lemma]{Lemma}
\aliascntresetthe{lemma}

\newaliascnt{corollary}{theorem}
\newtheorem{corollary}[corollary]{Corollary}
\aliascntresetthe{corollary}

\newaliascnt{conjecture}{theorem}
\newtheorem{conjecture}[conjecture]{Conjecture}
\aliascntresetthe{conjecture}

\newcommand{\mad}{\operatorname{mad}}
\newcommand{\og}{\operatorname{og}}

\crefname{theorem}{Theorem}{Theorems}
\crefname{lemma}{Lemma}{Lemmas}
\crefname{corollary}{Corollary}{Corollaries}
\crefname{conjecture}{Conjecture}{Conjectures}
\crefname{section}{Section}{Sections}

\hypersetup{
  pdftitle={A Proof of the Chen--Raspaud Conjecture},
  pdfauthor={Qi Wu and Yong Lu},
  pdfkeywords={Kneser graph, fractional coloring, maximum average degree, odd girth, discharging}
}

\title{A Proof of the Chen--Raspaud Conjecture\thanks{%
This work is supported by the National Natural Science Foundation of China
(Nos.~12371348 and 12201258) and the High-Quality Science and Technology
Cultivation Project of Jiangsu Normal University (No.~\mbox{JSNUGZL2026069}).}}
\author{Qi Wu, Yong Lu\thanks{Corresponding author.}\\[2pt]
\small School of Mathematics and Statistics, Jiangsu Normal University,\\[-1pt]
\small Xuzhou, Jiangsu 221116, People's Republic of China\\[-1pt]
\small Emails:~\texttt{wuqimath@163.com, luyong@jsnu.edu.cn}}
\date{}

\begin{document}
\maketitle

\begin{abstract}
For every integer $k\ge2$, Chen and Raspaud conjectured that each graph
$G$ with odd girth  $\og(G)\ge2k+1$ and maximum average degree $\mad(G)<2+1/k$ has a
$(2k+1:k)$-coloring. In this paper, we prove the conjecture.
\end{abstract}



\section{Introduction}

All graphs are finite and simple. For a positive integer $n$, write
$[n]=\{1,\ldots,n\}$. If $X$ is finite and $0\le r\le|X|$, then
$\binom{X}{r}$ is the family of all $r$-subsets of $X$. A graph
homomorphism from $G$ to $H$ is a map $f:V(G)\to V(H)$ such that
$f(u)f(v)\in E(H)$ whenever $uv\in E(G)$.

For positive integers $a\ge b$, an $(a:b)$-coloring of $G$ assigns a member
of $\binom{[a]}{b}$ to each vertex so that adjacent vertices receive
disjoint sets. Equivalently, it is a homomorphism from $G$ to the Kneser
graph $KG(a,b)$, whose vertices are the $b$-subsets of $[a]$ and whose
edges join disjoint sets. The problem goes back to Kneser~\cite{Kneser};
see also Hell and Ne\v{s}et\v{r}il~\cite{HellNesetril}. Stahl~\cite{Stahl}
systematically studied $n$-tuple colorings. The fractional chromatic number
is
\[
 \chi_f(G)=\inf\{a/b:G\text{ has an }(a:b)\text{-coloring}\};
\]
see Scheinerman and Ullman~\cite{ScheinermanUllman} and
Stahl~\cite{Stahl}.

For a graph $G$, its maximum average degree is
\[
 \mad(G)=\max_{\varnothing\ne H\subseteq G}
 \frac{2|E(H)|}{|V(H)|}.
\]
The maximum is taken over subgraphs with at least one vertex. The use of
subgraph density in coloring is closely related to the classical
degeneracy bound of Szekeres and Wilf~\cite{SzekeresWilf}. The girth
$g(G)$ is the length of a shortest cycle, with $g(G)=\infty$ for a forest.
The odd girth $\og(G)$ is the length of a shortest odd cycle, with
$\og(G)=\infty$ for a bipartite graph.

Chen and Raspaud~\cite{ChenRaspaud} proposed the following conjecture in
their study of homomorphisms to the Petersen graph.

\begin{conjecture}[Chen--Raspaud~\cite{ChenRaspaud}]\label{conj:CR}
For every integer $k\ge2$, if a graph $G$ satisfies
$\og(G)\ge2k+1$ and $\mad(G)<2+1/k$, then $G$ has a
$(2k+1:k)$-coloring.
\end{conjecture}

The target is $KG(2k+1,k)$. Lov\'asz~\cite{Lovasz} proved that
$\chi(KG(2k+1,k))=3$, and Denley~\cite{Denley} determined its odd girth:
$\og(KG(2k+1,k))=2k+1$. The graph is vertex-transitive. Hence the
Erd\H{o}s--Ko--Rado theorem~\cite{ErdosKoRado}, together with the standard
formula for the fractional chromatic number of a vertex-transitive graph
\cite{ScheinermanUllman}, gives
\[
 \chi_f(KG(2k+1,k))
 =\frac{\binom{2k+1}{k}}{\binom{2k}{k-1}}
 =\frac{2k+1}{k}.
\]
Thus the two hypotheses in \cref{conj:CR} match the odd girth and the
fractional chromatic number of the target.

The case $k=1$ follows from the greedy coloring argument: if
$\mad(G)<3$, every nonempty subgraph has a vertex of degree at most two,
so $G$ is $3$-colorable. Chen and Raspaud~\cite{ChenRaspaud} proved the
case $k=2$. The case $k=3$ was proved by {\L}yczek, Nazarczuk, and
Rz\k{a}\.{z}ewski~\cite{LyczekNazarczukRzazewski}. Choi~\cite{Choi}
proved the case $k=4$ and developed the full-thread replacement framework
used below. The path-extension criterion is due to
Klostermeyer and Zhang~\cite[Lemma~2.3]{KlostermeyerZhang}.

Related homomorphism and circular-coloring results for sparse graphs were
obtained by Borodin et al.~\cite{BorodinHartkeIvanovaKostochkaWest,
BorodinKimKostochkaWest}. Fractional coloring under degree, girth, and
planarity assumptions was studied in
\cite{DvorakHu,DvorakSereniVolec,DvorakSkrekovskiValla,HatamiZhu,
KardosKralVolec}. Cranston and West~\cite{CranstonWest} give an
introduction to the discharging method.

For $x,y\in\{0,1\}^m$, the Hamming distance $d_H(x,y)$ is the number of
coordinates on which they differ. A Hamming ball consists of the vectors
within a fixed Hamming distance of a center. Alon, Jin, and
Sudakov~\cite{AlonJinSudakov} determined the Helly number of equal-radius
Hamming balls. Junnila, Laihonen, Lehtil\"a, and Padavu
Devaraj~\cite{JunnilaLaihonenLehtilaDevaraj} gave an exact formula for
intersections of several $q$-ary Hamming balls with varying radii. The
hypothesis of \cref{lem:hamming} also bounds the total demand, so these
results do not apply directly.

Fiedorowicz's Lemma~3.1 in~\cite{Fiedorowicz} would imply a
homomorphism $KG(5,2)\to KG(9,4)$. This is impossible. The image of an odd
cycle is an odd closed walk of the same length, and every odd closed walk
contains an odd cycle no longer than the walk. Hence a homomorphism
$H_1\to H_2$ implies $\og(H_2)\le\og(H_1)$, whereas Denley's
formula~\cite{Denley} gives $\og(KG(5,2))=5$ and
$\og(KG(9,4))=9$. Consequently, the argument in that version does not establish
\cref{conj:CR}. Our proof does not use induction on $k$.

We prove the conjecture in full.

\begin{theorem}\label{thm:main}
Let $k\ge2$. If a graph $G$ satisfies $\og(G)\ge2k+1$ and
$\mad(G)<2+1/k$, then $G$ has a $(2k+1:k)$-coloring.
\end{theorem}

The graph-theoretic part starts with results already available in the
literature. We quote the path criterion of Klostermeyer and
Zhang~\cite[Lemma~2.3]{KlostermeyerZhang}, and Choi's minimum-counterexample
structure, star replacement, and triangular replacement
\cite[Lemmas~3.1, 3.2, 3.4, and~3.5]{Choi}. These results are not reproved.
We only verify the triangle--star extension needed when the positive
constraints consist of a triangle together with a star.

The set-theoretic part has two auxiliary results. The first is a common-point
lemma for Hamming balls under a total-demand condition. The second is a
weighted pair-cover inequality. Together they give the constant-weight quota
theorem. Choi's replacements and that theorem yield the local reducibility
statement, and a standard discharging argument completes the proof; see
Cranston and West~\cite{CranstonWest} for background on the method.

\section{Threads and replacements}\label{sec:threads}

Throughout the proof, fix $k\ge2$ and put $U=[2k+1]$. A \emph{color} is a
member of $\binom{U}{k}$. The length of a path is its number of edges. A
\emph{$t$-thread} is a path of length $t$ whose internal vertices have
degree two. A thread is \emph{maximal} if it is not contained in a longer
thread. A vertex of degree at least three is a \emph{$3^+$-vertex}. A
maximal thread is \emph{strong} if its ends are distinct $3^+$-vertices.
A \emph{pendent cycle} is a cycle containing exactly one $3^+$-vertex.

For $S,X\in\binom{U}{k}$, the pair $(S,X)$ is \emph{$t$-compatible} if
\begin{equation}\label{eq:compatible}
 \begin{cases}
 |S\cap X|\le (t-1)/2, & t\text{ is odd},\\[1mm]
 |S\cap X|\ge k-t/2, & t\text{ is even}.
 \end{cases}
\end{equation}

The following criterion is due to Klostermeyer and
Zhang~\cite[Lemma~2.3]{KlostermeyerZhang}.

\begin{lemma}[Klostermeyer--Zhang~\cite{KlostermeyerZhang}]\label{lem:thread}
Two prescribed end colors extend to a coloring of a $t$-thread if and only
if they are $t$-compatible. In particular, every pair of colors is
$t$-compatible when $t\ge2k$.
\end{lemma}

Let $\eta(G)$ be the number of $3^+$-vertices of $G$. Assume that
\cref{thm:main} is false. Following Choi~\cite[Section~3]{Choi}, choose a
counterexample $G_k$ by the following four-stage rule: first minimize
$\eta(G_k)$; subject to this, minimize the number of strong threads; then
maximize the sum of their lengths; and finally minimize $|E(G_k)|$.
Repeatedly delete vertices of degree one and choose an uncolorable component
of the remaining graph. Denote this component by $G_0$. As in Choi's
construction, $G_0$ is connected, has minimum degree at least two, and is not
a cycle. We shall use the first stage of the choice in the following form:
any graph satisfying the two hypotheses of \cref{thm:main} and having fewer
$3^+$-vertices than $G_k$ is colorable.

Choi's structural reductions apply directly to this choice
\cite[Lemmas~3.1 and~3.2]{Choi}.

\begin{lemma}[Choi's structural reductions~\cite{Choi}]\label{lem:structure}
The graph $G_0$ has no $2k$-thread and no pendent cycle.
\end{lemma}

Consequently, every maximal thread has length at most $2k-1$. Moreover,
every degree-two vertex lies on a unique strong maximal thread: extend the
corresponding degree-two chain in both directions; connectedness and
$\delta(G_0)\ge2$ force two $3^+$-ends, and the ends are distinct because
$G_0$ has no pendent cycle.

For a graph $H$, define its \emph{potential} by
\[
 \rho(H)=(2k+1)|V(H)|-2k|E(H)|.
\]
Since $\rho(H)$ is an integer, $\mad(G_0)<(2k+1)/k$ is equivalent to
$\rho(H)\ge1$ for every nonempty subgraph $H\subseteq G_0$.

Let $v\in V(G_0)$ have degree $d\ge3$. Let $T_1,\ldots,T_d$ be the strong
maximal threads incident with $v$, let $t_i=|E(T_i)|$, and let $x_i$ be
the other end of $T_i$. The vertices $x_i$ need not be distinct. After
$v$ and the internal vertices of these threads are deleted, the remaining
vertices are called \emph{old vertices}. A \emph{fresh link} of length
$L$ between two old vertices is a path of length $L$ whose internal
vertices are new. If its two old ends coincide, it is a cycle of length
$L$ meeting the old graph only at that vertex. Different fresh links have
disjoint sets of new internal vertices.

A fresh link adds $L-1$ vertices and $L$ edges, so its contribution to the
potential is $L-(2k+1)$. A link with coincident old ends has length at least
three. Indeed, a link of length two would come from two distinct
length-one threads with the same ends, which is impossible in a simple
graph.

The next two lemmas are the full-maximal-thread forms of Choi's star
and triangular replacements. We quote them without proof. The first is
Choi~\cite[Lemma~3.4]{Choi}.

\begin{lemma}[Choi's star replacement~\cite{Choi}]\label{lem:star}
Delete $v$ and the internal vertices of $T_1,\ldots,T_d$. For every
$i\in\{2,\ldots,d\}$, add a fresh link of length $t_1+t_i$ between $x_1$
and $x_i$, and let $R$ be the resulting graph. Then
\[
 \og(R)\ge\og(G_0),\qquad \mad(R)<2+\frac1k,
\]
and $R$ has a $(2k+1:k)$-coloring.
\end{lemma}

The second is Choi~\cite[Lemma~3.5]{Choi}.

\begin{lemma}[Choi's triangular replacement~\cite{Choi}]\label{lem:triangle}
Suppose that $d=3$ and $t_1+t_2+t_3>2k+1$. Delete $v$ and the internal
vertices of $T_1,T_2,T_3$, and add a fresh link of length $t_i+t_j$ between
$x_i$ and $x_j$ for every $1\le i<j\le3$. If $R$ is the resulting graph,
then
\[
 \og(R)\ge\og(G_0),\qquad \mad(R)<2+\frac1k,
\]
and $R$ has a $(2k+1:k)$-coloring.
\end{lemma}

The repeated-end convention above is a minor extension of Choi's statements.
Replacing a used fresh cycle by the two old arms through $v$ preserves length
and parity. Moreover, the numbers of new vertices and edges in a fresh link,
and the numbers of vertices and edges restored with any indexed family of
old arms, do not depend on whether some retained far ends coincide. Hence the
potential identities in Choi's proofs are unchanged. The odd-girth and
maximum-average-degree arguments in
\cite[Lemmas~3.4 and~3.5]{Choi} therefore apply under this convention.

For the local reduction we also need one extension of the triangular
replacement. The proof follows Choi's odd-walk and potential arguments
\cite[Lemmas~3.4 and~3.5]{Choi}, but the replacement graph has additional
leaves at one vertex of the triangle.

\begin{lemma}[Triangle--star extension]\label{lem:triangle-star}
Assume that $d\ge4$ and $t_1+t_2+t_3>2k+1$. Let $\Lambda$ be the graph on
$[d]$ with edge set
\[
 E(\Lambda)=\{12,13,23\}\cup\{1i:4\le i\le d\}.
\]
Delete $v$ and the internal vertices of all $T_i$. For every
$ij\in E(\Lambda)$, add a fresh link of length $t_i+t_j$ between $x_i$ and
$x_j$, and let $R$ be the resulting graph. Then
\[
 \og(R)\ge\og(G_0),\qquad \mad(R)<2+\frac1k,
\]
and $R$ has a $(2k+1:k)$-coloring.
\end{lemma}

\begin{proof}
The convention for coincident far ends is the one fixed above. As in Choi's
replacement proofs~\cite[Lemmas~3.4 and~3.5]{Choi}, replace every fresh link
on an odd cycle of $R$ by the two old threads through $v$. This preserves
length and parity and produces an odd closed walk in $G_0$. Hence
$\og(R)\ge\og(G_0)$.

Let $F$ be a nonempty subgraph of $R$. Repeatedly delete a fresh internal
vertex whose current degree is at most one, and let $K$ be the graph left.
Deleting a vertex of current degree $e\in\{0,1\}$ lowers the potential by
$(2k+1)-2ke\ge1$, so $\rho(F)\ge\rho(K)$, with strict inequality if
$K$ is empty. Every fresh link surviving in $K$ occurs in full. If $K$ is
empty, then $\rho(F)\ge1$. If $K$ is nonempty and has no fresh link, then
$K\subseteq G_0$, and again $\rho(F)\ge\rho(K)\ge1$.

Suppose that $K$ contains a fresh link. Let $\Lambda'\subseteq\Lambda$ be
the subgraph formed by the surviving links. Delete their edges and fresh
internal vertices from $K$, retaining the old ends, and call the remaining
graph $B$. Let $H\subseteq G_0$ be obtained from $B$ by adding $v$ and the
old threads indexed by $V(\Lambda')$. If $d_{\Lambda'}(i)$ denotes the
degree of $i$ in $\Lambda'$, then the same bookkeeping used by
Choi~\cite[Lemmas~3.4 and~3.5]{Choi} gives
\begin{align}
 \rho(K)-\rho(B)
 &=\sum_{i\in V(\Lambda')}d_{\Lambda'}(i)t_i
   -(2k+1)|E(\Lambda')|,\notag\\
 \rho(H)-\rho(B)
 &=\sum_{i\in V(\Lambda')}t_i
   -(2k+1)(|V(\Lambda')|-1).\notag
\end{align}
Therefore
\begin{equation}\label{eq:triangle-star-diff}
 \rho(K)-\rho(H)
 =\sum_{i\in V(\Lambda')}(d_{\Lambda'}(i)-1)t_i
 -(2k+1)(|E(\Lambda')|-|V(\Lambda')|+1).
\end{equation}
If $\Lambda'$ is a forest with $c$ nonempty components, then
$|E(\Lambda')|-|V(\Lambda')|+1=1-c\le0$, while each coefficient
$d_{\Lambda'}(i)-1$ is nonnegative. If $\Lambda'$ contains a cycle, then it
contains the triangle on $\{1,2,3\}$; every other edge is incident with
vertex $1$, so $\Lambda'$ is connected and its cyclomatic number is one.
The first sum in \eqref{eq:triangle-star-diff} is then at least
$t_1+t_2+t_3>2k+1$. Thus $\rho(K)\ge\rho(H)\ge1$ in all cases. Hence every
nonempty subgraph of $R$ has positive potential, and
$\mad(R)<2+1/k$.

The replacement removes the $3^+$-vertex $v$ and creates no new
$3^+$-vertex. Thus $R$ has fewer $3^+$-vertices than $G_k$. By the first
stage of Choi's minimum-counterexample choice~\cite[Section~3]{Choi}, $R$ is
colorable.
\end{proof}

\section{Set-system lemmas}\label{sec:setlemmas}

We now translate the extension problem at one vertex into a system of set
quotas. Let the incident maximal threads have lengths
$t_1,\ldots,t_d$, where $1\le t_i\le2k-1$. Let
$p_i\in\{0,1\}$ be defined by $p_i\equiv t_i\pmod 2$, and put
$q_i=k-\lfloor t_i/2\rfloor$. Then $1\le q_i\le k$ and
$t_i=2(k-q_i)+p_i$.

Suppose that the far end of the $i$th thread has color $X_i$. For
$A\subseteq U$, write $\mathbf 1_A$ for its characteristic vector. Put
$A_i=X_i$ when $p_i=0$, and put $A_i=U\setminus X_i$ when $p_i=1$.
Thus $|A_i|=k+p_i$. By \cref{lem:thread}, a color
$S\in\binom{U}{k}$ extends over the $i$th thread if and only if
\begin{equation}\label{eq:quotaform}
 |S\cap A_i|\ge q_i.
\end{equation}
If $X_i$ and $X_j$ are $(t_i+t_j)$-compatible, then
\begin{equation}\label{eq:distancebound}
 d_H(\mathbf 1_{A_i},\mathbf 1_{A_j})\le t_i+t_j.
\end{equation}
Indeed, when $p_i=p_j$, this is the even compatibility inequality in
\eqref{eq:compatible}; when $p_i\ne p_j$, it is the odd compatibility
inequality. Since $|A_i|=k+p_i$, \eqref{eq:distancebound} gives
\begin{equation}\label{eq:pairoverlap}
 |A_i\cap A_j|
 =\frac{|A_i|+|A_j|-d_H(\mathbf 1_{A_i},\mathbf 1_{A_j})}{2}
 \ge q_i+q_j-k.
\end{equation}

The same parameters describe the configurations that would receive too
little charge in the final discharging argument. Every vertex starts with
its degree, and each vertex of degree at least three sends $1/(2k)$ to
each internal degree-two vertex on an incident maximal thread. By
\cref{lem:structure}, a degree-$d$ vertex is incident with exactly $d$
maximal threads when they are counted by their first edges. If their
lengths are $t_1,\ldots,t_d$, its final charge is
\begin{equation}\label{eq:charge}
 d-\frac1{2k}\sum_{i=1}^d(t_i-1).
\end{equation}
This number is less than $2+1/k$ if and only if
\begin{equation}\label{eq:deficient}
 \sum_{i=1}^d t_i\ge(2k+1)d-4k-1.
\end{equation}
A list satisfying \eqref{eq:deficient} is called \emph{deficient}. Since
$t_i=2(k-q_i)+p_i$, condition \eqref{eq:deficient} is equivalent to
\begin{equation}\label{eq:totalquota}
 2\sum_{i=1}^d q_i+|\{i:p_i=0\}|\le4k+1.
\end{equation}
In particular, $d\le\sum_i q_i\le2k$.

Only the inequalities with $q_i+q_j>k$ require a replacement. The
\emph{constraint graph} is the graph $J$ on $[d]$ in which
$ij\in E(J)$ if and only if $q_i+q_j>k$. It has no two disjoint edges,
for two such edges would involve four indices whose $q_i$-values have sum
greater than $2k$. Hence the nonisolated part of $J$ is a star or a
triangle. To see this, take two adjacent edges $ab$ and $ac$. Every edge
avoiding $a$ must meet both of them, and therefore must be $bc$.

We need two set-theoretic lemmas. For $x\in\{0,1\}^m$, its
\emph{weight} is the number of coordinates equal to one. The first lemma
gives a common point for Hamming balls of different radii.

\begin{lemma}[Hamming-ball lemma]\label{lem:hamming}
Let $x_1,\ldots,x_d\in\{0,1\}^m$, and let $1\le a_i\le m$. Assume that
\begin{align}
 d_H(x_i,x_j)&\le2m-a_i-a_j &&(i\ne j),\label{eq:hpair}\\
 \sum_{i=1}^d a_i&\le2m-1.\label{eq:htotal}
\end{align}
Then there is a vector $y\in\{0,1\}^m$ such that
$d_H(y,x_i)\le m-a_i$ for every $i$.
\end{lemma}

\begin{proof}
We use induction on $m$. The cases $d\le1$ are immediate. If
some $a_i=m$, take $y=x_i$; by condition \eqref{eq:hpair},
$d_H(x_i,x_j)\le m-a_j$ for every $j$. If $d=2$, put $r_i=m-a_i$ and
$D=d_H(x_1,x_2)$. By \eqref{eq:hpair}, $D\le r_1+r_2$. A shortest path
from $x_1$ to $x_2$ in the Hamming cube contains a point at distance
$\min\{r_1,D\}$ from $x_1$, and this point is at distance at most $r_2$
from $x_2$.

Assume that $d\ge3$ and $a_i\le m-1$ for all $i$. Fix a coordinate $r$.
For distinct $i,j$, put
\[
 \sigma_{ij}=2m-a_i-a_j-d_H(x_i,x_j)\ge0.
\]
For $b\in\{0,1\}$, let $I_b=\{i:(x_i)_r=b\}$. At least one of
$I_0$ and $I_1$ contains no pair $ij$ with $\sigma_{ij}\le1$.
Otherwise, choose such a pair in each class. The two pairs are disjoint,
and the two vectors in either pair agree at coordinate $r$, so their
distance is at most $m-1$. It follows that $a_i+a_j\ge m$ for each
pair, contrary to \eqref{eq:htotal}.

We now fix the $r$th coordinate of the desired vector. If it is fixed to
$b$, delete coordinate $r$ and define
\[
 a_i'=\begin{cases}
 a_i-1,&i\in I_b,\\
 a_i,&i\notin I_b.
 \end{cases}
\]
An index with $a_i'=0$ may be discarded, since its reduced ball is the
whole $(m-1)$-dimensional cube. The pairwise slack in the reduced system is
unchanged unless both indices lie in $I_{1-b}$; in that case it is
$\sigma_{ij}-2$. We therefore choose $b$ so that $I_{1-b}$ contains no
pair $ij$ with $\sigma_{ij}\le1$. The preceding paragraph shows that such
a value exists.

Let $A=\sum_i a_i$. We may also require
$A-|I_b|\le2m-3$. This is immediate when $A\le2m-3$. Suppose that
$A=2m-2$. If the chosen class is nonempty, there is nothing to prove. If
it is empty, reverse the value of $b$; the new unchosen class is empty and
the new chosen class is nonempty. Now suppose that $A=2m-1$. If
$|I_b|\ge2$, there is again nothing to prove. Otherwise reverse $b$. The
new unchosen class has at most one index, while the new chosen class has
at least two indices because $d\ge3$. Thus the reduced demands satisfy
\[
 \sum_i a_i'=A-|I_b|\le2m-3=2(m-1)-1.
\]

The induction hypothesis now gives a vector in the reduced cube; if no
positive demand remains, choose any reduced vector. Restore coordinate
$r$ with value $b$. For $i\in I_b$, this coordinate supplies the one
agreement removed from $a_i$, while for $i\notin I_b$ it supplies no
agreement and the demand was unchanged. The restored vector satisfies all
original inequalities.
\end{proof}

The vector supplied by \cref{lem:hamming} need not have the required
weight. We use the following counting bound to correct it. The inequality
is related to lower bounds for block sizes in pairwise balanced designs and
sigma clique partitions; see Davoodi, Javadi, and
Omoomi~\cite{DavoodiJavadiOmoomi}. A family
$B_1,\ldots,B_N$ \emph{covers every pair} of $\Omega$ if every member of
$\binom{\Omega}{2}$ is contained in at least one $B_i$.

\begin{lemma}\label{lem:paircover}
Let $|\Omega|=v\ge2$ and $s\ge1$. Let
$B_1,\ldots,B_N\subsetneq\Omega$ cover every pair of points of $\Omega$.
Let $\varepsilon_i\in\{0,1\}$ and assume
\begin{equation}\label{eq:pcsize}
 s+\varepsilon_i\le |B_i|\le v-s-1+\varepsilon_i
 \qquad(1\le i\le N).
\end{equation}
Then
\begin{equation}\label{eq:pcbound}
 \sum_{i=1}^N(2|B_i|-s-\varepsilon_i)\ge4v-3(s+1).
\end{equation}
\end{lemma}

\begin{proof}
For $x\in\Omega$, let $I_x=\{i:x\in B_i\}$. Since the $B_i$ cover every
pair, $I_x\cap I_y\ne\varnothing$ whenever $x\ne y$. Moreover, since each $B_i$ is a proper subset of $\Omega$, we have
$\bigcap_{x\in\Omega}I_x=\varnothing$. If
$I_x=\{i\}$, then $I_x\cap I_y\ne\varnothing$ forces $i\in I_y$
for every $y$. Hence $B_i=\Omega$, a contradiction. Thus $|I_x|\ge2$
for every $x$.

Let $e=\sum_{x\in\Omega}(|I_x|-2)$ and
$P=\sum_{i=1}^N\varepsilon_i$. Since
$\sum_i|B_i|=\sum_x|I_x|=2v+e$, inequality \eqref{eq:pcbound} is equivalent
to
\begin{equation}\label{eq:pctarget}
 s(N-3)+P-3\le2e.
\end{equation}
Since $|I_x|\ge2$, $N=1$ is impossible. If $N=2$, then $I_x=\{1,2\}$ for every $x$, a contradiction to
$\bigcap_x I_x=\varnothing$. Thus $N\ge3$.

We now consider the distinct two-element sets among the $I_x$ as edges of a simple graph on $[N]$. These edges meet pairwise, so they form a star or a triangle.
The number of points having the same set $I_x$ is still counted in all sums.

First suppose that these edges do not form a star with at least three leaves. We choose a three-set $T\subseteq[N]$. If the edges form a triangle, take its vertex set. If they form a star with one leaf, take the
two vertices of its unique edge and one further index. If they form a star with two leaves, take the center and the two leaves. Then every two-element
set $I_x$ is contained in $T$, and every three-element set $I_x$ meets $T$. If there is no two-element set, take a three-element set $I_x$ as $T$
when one exists; if all $I_x$ have size at least four, take any three indices. In all cases,
$|I_x\setminus T|\le2(|I_x|-2)$ for every $x$. Since
$\sum_{i\in T}\varepsilon_i\le3$, by the lower bound in
\eqref{eq:pcsize},
\[
 s(N-3)+P-3
 \le\sum_{i\notin T}|B_i|
 =\sum_x|I_x\setminus T|
 \le2e.
\]

It remains to consider a star with center $c$ and leaf set $L$, where
$r=|L|\ge3$. Let $O=[N]\setminus(\{c\}\cup L)$ and $|O|=o$. We split the points of $\Omega$ into three types: those with $|I_x|=2$, those with
$c\notin I_x$ and $|I_x|\ge3$, and those with $c\in I_x$ and
$|I_x|\ge3$. Every set of the second type contains all of $L$, since it must meet every edge $\{c,\ell\}$ of the star.

Let $n_0$ be the number of points with $c\notin I_x$ and $|I_x|\ge3$,
and let $n_1$ be the number with $c\in I_x$ and $|I_x|\ge3$. Let $L_1$ be the number of incidences between the latter points and $L$, and let $D_O$ be the total number of incidences with $O$. A two-element pattern
$\{c,\ell\}$ contributes zero to $e$. A point counted by $n_0$ contains all
$r$ leaves and contributes $r-2+|I_x\cap O|$. A point counted by $n_1$
contributes $|I_x\cap L|+|I_x\cap O|-1$. Therefore
$
 e=n_0(r-2)+L_1+D_O-n_1.
$
Every two-element set $I_x$ contains $c$. Hence a point not in $B_c$
cannot have $|I_x|=2$, and therefore $n_0=v-|B_c|$. By
\eqref{eq:pcsize},
\begin{equation}\label{eq:nzero}
 n_0\ge s+1-\varepsilon_c.
\end{equation}
Let $P_O=\sum_{i\in O}\varepsilon_i$ and
$P_L=\sum_{i\in L}\varepsilon_i$. Then
$D_O=\sum_{i\in O}|B_i|\ge so+P_O$, and every point counted by $n_1$ has at
least two coordinates besides $c$, so $L_1+D_O\ge2n_1$. Since
$N=1+r+o$ and $P=\varepsilon_c+P_L+P_O$, we have
\begin{align*}
&2e-[s(N-3)+P-3]\\
&\quad=(2n_0-s)(r-2)+2L_1+2D_O-2n_1
      -so-\varepsilon_c-P_L-P_O+3\\
&\quad\ge(2n_0-s)(r-2)+2L_1+D_O-2n_1
      -\varepsilon_c-P_L+3\\
&\quad\ge(2n_0-s)(r-2)-\varepsilon_c-r+3\\
&\quad=(2n_0-s-1)(r-2)+(1-\varepsilon_c)\ge0.
\end{align*}
The last inequality follows from \eqref{eq:nzero}: if $\varepsilon_c=0$,
then $n_0\ge s+1$; if $\varepsilon_c=1$, then $n_0\ge s$. Thus
\eqref{eq:pctarget} holds.
\end{proof}

\section{Constant-weight quotas and local reduction}\label{sec:quota}

The two lemmas combine as follows. For constraints of the form
$d_H(y,x_i)\le t_i$, the \emph{slack} of index $i$ at $y$ is
$t_i-d_H(y,x_i)$. The index is \emph{tight} when this slack is zero.

\begin{theorem}\label{thm:quota}
Let $|U|=2k+1$. For $i\in[d]$, let $A_i\subseteq U$,
$|A_i|=k+p_i$, $p_i\in\{0,1\}$, and $1\le q_i\le k$. Assume that
\begin{align}
 |A_i\cap A_j|&\ge q_i+q_j-k &&(i\ne j),\label{eq:qpair}\\
 2\sum_{i=1}^d q_i+|\{i:p_i=0\}|&\le4k+1.\label{eq:qtotal}
\end{align}
Then there is a set $S\in\binom{U}{k}$ such that $|S\cap A_i|\ge q_i$ for
every $i$.
\end{theorem}

\begin{proof}
Let $m=2k+1$, let $x_i$ be the characteristic vector of $A_i$, and put
$t_i=2(k-q_i)+p_i$ and $a_i=m-t_i=2q_i+1-p_i$. By
\eqref{eq:qpair},
\[
 d_H(x_i,x_j)\le t_i+t_j=2m-a_i-a_j.
\]
Moreover, by \eqref{eq:qtotal}, we have $\sum_i a_i\le4k+1=2m-1$.
Thus by \cref{lem:hamming}, there is a vector $y\in\{0,1\}^m$ such that $d_H(y,x_i)\le t_i$ for every $i$.

Among all such vectors, choose $y$ so that $\bigl||Y|-k\bigr|$ is minimum,
where $Y=\{r:y_r=1\}$. For each $i$, the slack equals
\begin{equation}\label{eq:slack}
 t_i-d_H(y,x_i)=k-|Y|+2(|Y\cap A_i|-q_i).
\end{equation}
All slacks therefore have the same parity as $|Y|-k$. If this difference
were odd, every slack would be at least one. Flipping one coordinate in the
direction of weight $k$ increases each Hamming distance by at most one and
would preserve feasibility, contrary to the choice of $y$. Hence
$|Y|=k+2h$ or $|Y|=k-2h$ for some $h\ge0$. It remains to show that $h=0$.

\medskip
\noindent\textbf{Case 1.} $|Y|=k+2h$ and $h\ge1$.
Let $I=\{i:d_H(y,x_i)=t_i\}$. A nontight slack is a positive even integer, so it is at
least two. For $D\in\binom{Y}{2}$, let $y'$ be the characteristic vector of
$Y\setminus D$.
Then
\begin{equation}\label{eq:deletepair}
 d_H(y',x_i)-d_H(y,x_i)=2|D\cap A_i|-2.
\end{equation}
The vector $y'$ has weight closer to $k$, so it is not feasible. By
\eqref{eq:deletepair}, a nontight condition cannot fail, and a tight
condition fails exactly when $D\subseteq A_i$. Thus the sets
$B_i=Y\cap A_i$, $i\in I$, cover all pairs of $Y$.

For $i\in I$, equality in \eqref{eq:slack} gives $|B_i|=q_i+h$. Since
$B_i\subseteq A_i$, we have $q_i\le k+p_i-h$. Also,
$|A_i\setminus Y|=k+p_i-q_i-h\le|U\setminus Y|=k+1-2h$, and hence
\begin{equation}\label{eq:highbounds}
 h+p_i-1\le q_i\le k+p_i-h.
\end{equation}
Apply \cref{lem:paircover} with $v=k+2h$, $s=2h-1$, and
$\varepsilon_i=p_i$. By \eqref{eq:highbounds},
$s+p_i\le|B_i|\le v-s-1+p_i$. Moreover, $|B_i|<v$ since
$2h-p_i\ge1$. Therefore
\[
 \sum_{i\in I}(2q_i+1-p_i)
 =\sum_{i\in I}(2|B_i|-s-p_i)
 \ge4(k+2h)-6h=4k+2h>4k+1,
\]
a contradiction to \eqref{eq:qtotal}.

\medskip
\noindent\textbf{Case 2.} $|Y|=k-2h$ and $h\ge1$.
Let $Z=U\setminus Y$. For $D\in\binom{Z}{2}$, let $y'$ be the
characteristic vector of $Y\cup D$. Then
\[
 d_H(y',x_i)-d_H(y,x_i)=2-2|D\cap A_i|.
\]
Let $I=\{i:d_H(y,x_i)=t_i\}$, and put
$B_i=Z\cap(U\setminus A_i)$ for $i\in I$. These sets cover every pair of
$Z$. Otherwise, some $D\in\binom{Z}{2}$ would be contained in no $B_i$ with
$i\in I$. For every tight index we would then have $|D\cap A_i|\ge1$, so
the displayed formula would not increase the corresponding distance. A
nontight condition has slack at least two and would also remain feasible.
Thus $y'$ would be feasible and closer to weight $k$, a contradiction.

Tightness in \eqref{eq:slack} gives
$|Y\cap A_i|=q_i-h$ and
$|B_i|=q_i+h+1-p_i$. Hence $q_i\ge h$; since
$B_i\subseteq U\setminus A_i$, we also have $q_i\le k-h$. Apply
\cref{lem:paircover} with $v=k+1+2h$, $s=2h$, and
$\varepsilon_i=1-p_i$. Indeed,
\[
 2h+1-p_i\le q_i+h+1-p_i=|B_i|\le k+1-p_i
 =v-s-1+\varepsilon_i.
\]
Also $|B_i|\le k+1-p_i<k+1+2h=|Z|$, so every $B_i$ is proper. Therefore
\[
 \sum_{i\in I}(2q_i+1-p_i)
 =\sum_{i\in I}(2|B_i|-s-(1-p_i))
 \ge4(k+1+2h)-3(2h+1)
 =4k+1+2h>4k+1,
\]
which again contradicts \eqref{eq:qtotal}.

Hence $h=0$ and $|Y|=k$. For a weight-$k$ vector,
$d_H(y,x_i)\le t_i$ is equivalent to $|Y\cap A_i|\ge q_i$. Take $S=Y$.
\end{proof}

We now combine Choi's replacements~\cite[Lemmas~3.4 and~3.5]{Choi},
the triangle--star extension, and \cref{thm:quota}.

\begin{corollary}\label{cor:local}
No $3^+$-vertex of $G_0$ has a deficient list of incident thread lengths.
\end{corollary}

\begin{proof}
Suppose that a vertex $v$ has a deficient list
$t_1,\ldots,t_d$. Put $p_i\equiv t_i\pmod 2$ and
$q_i=k-\lfloor t_i/2\rfloor$, and let $J$ be the graph on $[d]$ in which
$ij\in E(J)$ exactly when $q_i+q_j>k$. Condition \eqref{eq:qtotal}
follows from \eqref{eq:totalquota}. By the structure of
$J$, one of the following three cases occurs.

If $J$ has no edge, let $R$ be obtained by deleting $v$ and the
internal vertices of all incident threads. The graph $R$ satisfies the two
hypotheses of \cref{thm:main} and has fewer $3^+$-vertices than $G_k$.
Hence it is colorable by the first stage of Choi's minimum-counterexample
choice~\cite[Section~3]{Choi}.

If the nonisolated part of $J$ is a star, relabel the threads so that its
center is $1$ and apply Choi's star replacement
\cref{lem:star}. Every edge of $J$ is then represented by a fresh link of
length $t_i+t_j$.

Suppose finally that the nonisolated part of $J$ is the triangle on
$\{1,2,3\}$. If $d=3$, then \eqref{eq:deficient} gives
$t_1+t_2+t_3\ge2k+2$, so Choi's triangular replacement
\cref{lem:triangle} applies. If $d\ge4$, then
$q_1+q_2+q_3\le2k-d+3$, and hence
\[
 t_1+t_2+t_3
 =6k-2(q_1+q_2+q_3)+p_1+p_2+p_3
 \ge2k+2d-6\ge2k+2.
\]
We then apply the triangle--star extension
\cref{lem:triangle-star}. In both cases every edge of $J$ is represented by
a fresh link of length $t_i+t_j$.

Fix a $(2k+1:k)$-coloring of the graph $R$ obtained in the appropriate
case. Let $X_i$ be the color of $x_i$. Set $A_i=X_i$ when $p_i=0$, and set
$A_i=U\setminus X_i$ when $p_i=1$. If $ij\in E(J)$, the coloring of the
fresh link shows, by \cref{lem:thread}, that $X_i$ and $X_j$ are
$(t_i+t_j)$-compatible. Therefore \eqref{eq:pairoverlap} gives
$|A_i\cap A_j|\ge q_i+q_j-k$. If $ij\notin E(J)$, the same inequality is
automatic because its right side is nonpositive. Thus all assumptions of
\cref{thm:quota} hold.

By \cref{thm:quota}, choose $S\in\binom{U}{k}$ with
$|S\cap A_i|\ge q_i$ for every $i$. By \eqref{eq:quotaform}, the colors $S$ and $X_i$ extend over the $i$th old
thread. The incident threads have pairwise disjoint interiors, so the
extensions combine to a coloring of $G_0$, a contradiction.
\end{proof}

\section{Discharging and consequences}\label{sec:discharging}

\begin{proof}[Proof of \cref{thm:main}]
We work with Choi's minimum-counterexample graph $G_0$ described in
\cref{sec:threads}; its structural properties are those in
\cite[Lemmas~3.1 and~3.2]{Choi}. Give every vertex of $G_0$ its degree as
initial charge. Each $3^+$-vertex
sends $1/(2k)$ to every degree-two vertex lying internally on one of its
incident maximal threads.

By \cref{lem:structure}, every degree-two vertex lies on one strong maximal
thread. It receives $1/(2k)$ from each end, so its final charge is
$2+1/k$.

Let $v$ have degree $d\ge3$, and let $t_1,\ldots,t_d$ be the lengths of
its incident maximal threads. Its final charge is the quantity in
\eqref{eq:charge}. If $3\le d\le2k$ and this charge is less than
$2+1/k$, then \eqref{eq:deficient} holds, contrary to
\cref{cor:local}. If $d\ge2k+1$, then $t_i\le2k-1$ by
\cref{lem:structure}, and the final charge is at least
\[
 d-\frac{d(2k-2)}{2k}=\frac dk\ge2+\frac1k.
\]
Thus every vertex has final charge at least $2+1/k$.

The total charge is unchanged and equals $2|E(G_0)|$. Hence the average
degree of $G_0$ is at least $2+1/k$, contrary to
$\mad(G_0)<2+1/k$.
\end{proof}

The theorem has the following planar consequence.

\begin{corollary}\label{cor:planar}
Let $k\ge2$. Every planar graph of girth at least $4k+2$ has a
$(2k+1:k)$-coloring.
\end{corollary}

\begin{proof}
If $G$ is a forest, then $\mad(G)<2$ and $\og(G)=\infty$, so the result
follows from \cref{thm:main}. Otherwise let $g=g(G)\ge4k+2$. Euler's
formula gives
\[
 \mad(G)<\frac{2g}{g-2}\le2+\frac1k,
\]
and $\og(G)\ge g\ge2k+1$. Apply \cref{thm:main}.
\end{proof}

For $k=2$, Dvo\v{r}\'ak, \v{S}krekovski, and
Valla~\cite{DvorakSkrekovskiValla} proved that odd girth at least $9$
suffices; in particular, full girth at least $8$ suffices. For $k\ge3$,
Klostermeyer and Zhang~\cite[Theorem~1.1]{KlostermeyerZhang} obtained the
odd-girth bound $10k-7$, and Ren and Bu~\cite{RenBu} improved it to
$10k-9$. These results control only odd cycles, whereas
\cref{cor:planar} assumes a lower bound on the full girth.

\section*{Declaration of competing interest}

The authors declare that they have no known competing financial interests or
personal relationships that could have appeared to influence the work
reported in this paper.

\section*{Data availability}

No data were used for the research described in this article.

\section*{Declaration on the use of AI}

The authors used ChatGPT 5.6 Pro to assist in discussing proof strategies, checking proofs, and improving exposition.

\end{document}